\documentclass[12pt]{amsart}
\usepackage{amsthm,amsmath,a4,amssymb,verbatim,url,lscape}

\newtheorem{thm}{Theorem}[section]
\newtheorem{cor}[thm]{Corollary}
\newtheorem{lem}[thm]{Lemma}
\newtheorem{prop}[thm]{Proposition}
\theoremstyle{definition}
\newtheorem{defn}[thm]{Definition}
\theoremstyle{remark}
\newtheorem{rem}[thm]{Remark}

\numberwithin{equation}{section}

\newcommand{\Aut}{\textnormal{Aut}}

\newcommand{\N}{\mathbf{N}}

\newcommand{\Z}{\mathbf{Z}}

\newcommand{\C}{\mathbf{C}}

\newcommand{\Compl}{\mathrm{Compl}}
\newcommand{\Bir}{\mathrm{Bir}}

\newcommand{\dra}{\dashrightarrow}
\newcommand{\bfk}{\mathbf{k}}

\begin{document}

\address{CNRS and Univ Lyon, Univ Claude Bernard Lyon 1, Institut Camille Jordan, 43 blvd. du 11 novembre 1918, F-69622 Villeurbanne}
\email{cornulier@math.univ-lyon1.fr}

\title{Regularization of birational actions of FW groups}
\author{Yves Cornulier}%
\date{March 11, 2021}


\begin{abstract}
We prove that every birational action of a group with Property FW can be regularized.
\end{abstract}
\maketitle

\section{Result and context}

Fix a ground algebraically closed field $\bfk$ (see Remark \ref{ground_field} as regards arbitrary ground fields). By {\bf prevariety} we mean a scheme over $\bfk$, locally of finite type over $\bfk$. That is, a prevariety is a topological space $X$, endowed with a sheaf of $\bfk$-algebras, that is locally isomorphic to an affine variety (namely $\mathrm{Spec}(A)$ for $A$ finitely generated $\bfk$-algebra with its structural sheaf). A prevariety is {\bf separated} if the diagonal embedding $X\to X\times X$ is a closed immersion (the product $\times$ is over $\mathrm{Spec}(\bfk)$).
A {\bf variety} is a separated prevariety that is noetherian (equivalently, has a finite cover by affine open subsets). We do not assume it to be reduced.

For groups, {\bf Property FW} is a combinatorial restriction on its actions. It notably holds for groups with Kazhdan's Property~T such as $\mathrm{SL}_3(\Z)$ and its finite index subgroups; see \S\ref{sec_parta} for the definition. Here we prove the following, which solves \cite[Question 10.1]{CC}.

\begin{thm}\label{main_reg}
Let $X$ be an irreducible variety. 
Let $G$ be a group with Property FW and $G\to\Bir(X)$ a birational action. Then there exist an irreducible variety $Y$, a bi-regular action $G\to\Aut(Y)$, and a $G$-equivariant birational transformation $X\dra Y$.
\end{thm}

\begin{rem}\label{firstrem}In \cite{CC} in the case of (reduced) surfaces, a stronger assumption is obtained, namely that $Y$ can be chosen to be projective and smooth. However, this fails in dimension $\ge 3$: (at least) in characteristic zero the monomial copy of $\mathrm{SL}_3(\Z)$ in $\mathrm{Bir}(\mathbb{P}^3)$ cannot be regularized in any projective model, according to classification results of Cantat and Zeghib \cite{CZ} in the smooth case, and therefore in general, using equivariant resolution of singularities. On the other hand, if $X$ is a normal variety, then $Y$ can be chosen to be quasi-projective, see Remark \ref{normal_qp}.
\end{rem}

 Using results about finitely generated subgroups of automorphism groups of varieties \cite{BL}, this yields:

\begin{cor}
Let $G$ be a finitely generated group with Property FW and $X$ an irreducible variety. The following two assertions hold.
\begin{enumerate}
\item If $G\to\Bir(X)$ is an injective homomorphism, then $G$ is residually finite; if moreover $\mathbf{k}$ has characteristic zero, then $G$ is virtually torsion-free. 
\item If $G$ has no nontrivial finite quotient, then $\mathrm{Hom}(G,\Bir(X))$ is trivial.\qed
\end{enumerate}
\end{cor}

The assertions of the corollary (for $X$ reduced in characteristic zero) are known for groups with Kazhdan's Property T \cite{CX}, using analytic methods. Nevertheless, Theorem \ref{main_reg} is new also in the case when $G$ has Kazhdan's Propery T, say with $X$ smooth over $\C$.

In \cite{CC} a weaker regularization result was established. Namely it was proved (assuming $X$ reduced) that one can find such $Y$ with an action by ``pseudo-automorphisms", i.e., for which the birational maps are defined as isomorphisms between complements of closed subsets of codimension $\ge 2$. In a sense, this is the first step towards the above regularization theorem. Nevertheless, our approach, even in this first step, is different and makes use of Exel's notion of partial action. It allows to embed every variety $X$ as a dense open subset of a prevariety $\hat{X}$, so that $\Bir(X)$ naturally acts on $\hat{X}$ in a bi-regular way restricting to the given birational action on $X$ (see Proposition \ref{actcomplet}). At this point we reach a regularization result in a somewhat insane generality; the cost being that $\hat{X}$ is usually neither separated nor noetherian (it is obtained by gluing copies of $X$, indexed by $\Bir(X)$, over open subsets). The work then consists in using Property FW so as to find in $\hat{X}$ a dense open invariant subvariety.

There are two parts in the proof. The first part consists in finding an open noetherian $G$-invariant subset. The second part relies on a very general lemma: inside this noetherian open subset, one produces a $G$-invariant separated subset. The first part is where Property FW is used, and all the work is done inside a prevariety that is not generally separated nor noetherian (although it has a dense open affine subvariety). 

\begin{rem}
When considering birational groups, it is usual to assume irreducible varieties to be reduced, but this sounds often unnecessary. For $X=\mathrm{Spec}(A)$ for some finitely generated $\bfk$-algebra $A$ and arbitrary field $\bfk$, assuming that $X$ is irreducible means that the nilradical $R$ of $A$ is a prime ideal (the unique minimal one), and $\Bir(X)$ is naturally isomorphic to $\Aut_{\bfk\mbox{-alg}}(S^{-1}A)$, for $S=A\smallsetminus R$; here $S^{-1}A$ is the total ring of fractions. There is a canonical homomorphism $\Bir(X)\to\Bir(X_{\mathrm{red}})$ which is usually not injective, for instance for $A=\bfk[x,y,z]/(z^2)$ (in which case it is surjective). 
\end{rem}

\section{Discussion, main concepts, and restatement of the main result}

Main convention: for a variety $X$, we insist that we use the schematic point of view, for which all irreducible subvarieties are viewed as elements of $X$.

It is useful to give a precise definition of a birational transformation, as we do not want to define it as an equivalence class. Let $X,Y$ be irreducible (possibly not reduced) varieties. Namely, we define a {\bf birational transformation} $u:X\dra Y$ as a closed irreducible subvariety $f$ of $X\times Y$, such that there exist open dense subvarieties $U\subset X$, $U'\subset Y$, such that $f\cap (U\times U')$ is the graph of an isomorphism $U\to U'$. Then it is a standard verification that there exists a unique maximal such pair $(U_f,U'_f)$. Indeed, there exists a unique maximal open subset $V_f$ such that $f\cap (V_f\times Y)$ is the graph of a regular map $V_f\to Y$, and $U_f=\{x\in V_f:f(x)\in V_{f^{-1}}\}$, where for $x\in V_f$, the unique $y\in Y$ such that $(x,y)\in f$ is denoted by $f(x)$.

A basic fact is the following: if $X,Y,Z$ are irreducible varieties, with birational transformations $X\stackrel{f}\dra Y\stackrel{g}\dra Z$, the composition $g\circ f$ can be defined as the intersection over all dense open subsets $f',g'$ of $f$ and $g$ of the closure of $g'f'=\{(x,z)\in X\times Z:\exists y\in Y:(x,y)\in f,(y,z)\in g\}$. Then $U_{g\circ f}\supset U_f\cap f^{-1}(U_g\cap U_{f^{-1}})$. This implies, for fixed $X$, that mapping $f\in\Bir(X)$ to the partial bijection $f:U_f\to U_{f^{-1}}$ is a partial action in the sense of Exel~\cite{E}:

\begin{defn}
A {\bf partial action} of a group $G$ on a set $X$ is a map $\alpha$ from $G$ to the set of partial bijections of $X$: $\alpha(g):D_g\to D'_g$, satisfying: $\alpha(1)=\mathrm{id}_X$, $\alpha(g^{-1})=\alpha(g)^{-1}$, and $\alpha(gh)\supset\alpha(g)\alpha(h)$ for all $g,h\in G$. We call $X$ a {\bf partial $G$-set}.
\end{defn}

Here $D_g$ and $D'_g$ are called {\bf domain} and {\bf codomain} of $g$, and $D'_g=D_{g^{-1}}$. The last condition means, writing $gx=\alpha(g)x$, that whenever $hx$ and $g(hx)$ are defined, then $(gh)x$ is defined and equals $g(hx)$. 

Given a partial action of $G$ on $X$ and a subset $Y$, one obtains by restriction a partial action of $G$ on $Y$ (called {\bf restricted} partial action on $Y$), defined by $g\mapsto \alpha(g)\cap (Y\times Y)$. In particular, we can construct a partial action starting with an action and restrict to a subset (the resulting partial action being an action only when the subset is invariant under the action). This is actually the only way to produce partial actions:

\begin{prop}[Abadie, Kellendonk-Lawson]\label{abadie}
For every partial action $\alpha$ of a group $G$ on a set $X$, there exists a $G$-set $\hat{X}$ (also denoted $\Compl(G,X)$ or $\Compl(G,\alpha)$ if we need to emphasize $G$ or $\alpha$) and an injective map $X\to\hat{X}$, such that $X$ meets every $G$-orbit, and $\alpha$ is the restricted partial action. Such a $G$-set is unique up to a $G$-equivariant bijection inducing the identity on $X$.
\end{prop}

This $G$-set $\hat{X}$ is called {\bf universal globalization} of the partial $G$-set $X$. It is defined by considering $G\times X$ with $G$-action $g\cdot (h,x)=(gh,x)$, and identifying $(g,x)$ and $(h,y)$ whenever there exists $k\in G$ such that $(kg)x$ and $(kh)y$ are defined and equal. The set $X$ is embedded by mapping $x$ to the class of $(1,x)$. Verifications are direct. From uniqueness, we deduce that if $E$ is a $G$-set and $X$ a subset (thus viewed as partial $G$-set), then its universal globalization can be identified with the inclusion of $X$ into $GX$, the smallest $G$-invariant subset of $E$ containing~$X$.

 At first glance, this seems to mean that the notion of partial action is pointless, since it just consists in restricting actions to arbitrary subsets. But the point is that we deal in practice with explicit partial actions (for instance, of a group with a birational action on a variety) whose universal globalization is a quite huge mysterious and complicated object.

The universal globalization inherits local structures from $X$ that are preserved by the $G$-action.
The first is the topology: if a partial action is topological, i.e., if for every $g\in G$, the domain $D_g$ is open and $\alpha(g):D_g\to D'_g$ is a homeomorphism, then there is a unique topology on $\hat{X}$ for which $X$ is open, the induced topology on $X$ is the original one, and the $G$-action on $X$ is continuous (i.e., by self-homeomorphisms -- no topology is considered on $G$). Moreover, if $D_g$ is dense for every $g$, then $X$ is dense in $\hat{X}$. The verifications are straightforward and done in both \cite{AT,A} and \cite{KL}.

Second, the structure of prevariety is inherited by $\hat{X}$:

\begin{prop}\label{actcomplet}
Let $X$ be an irreducible variety. Let $G\to\Bir(X)$ be a homomorphism. Then then there exists a unique structure of prevariety on $\hat{X}=\Compl(G,X)$ for which $X$ is a dense open subset and $G$ acts by (bi-regular) automorphisms.
\end{prop}
(Unique means the topology is unique, and the prevariety structure is unique up to isomorphism of ringed space being the identity on the basis.)
\begin{proof}
Since the partial action of $G$ is topological, there is a unique topology on $\hat{X}$ making the $G$-action continuous. We have a system of charts, indexed by $g\in G$, given by the action of $g^{-1}$ on $gX$, valued in $X$. This carries the variety structure to $gX$, and the change of charts are by definition bi-regular. Whence the result.

Since domains of definitions are dense, $X$ is dense, as observed above.
\end{proof}

Of course this does not need to be a variety, i.e., does not have to be separated or noetherian (otherwise all birational actions could be regularized). However, this yields for free a regularization in the larger class of prevarieties, and Theorem \ref{main_reg} follows from the following:

\begin{thm}\label{sep_core}
Let $G$ be a group with Property FW and $G\to\Bir(X)$ a birational action. Then there exists a dense $G$-invariant open subset of $\hat{X}=\Compl(G,X)$ that is a variety.
\end{thm}

\begin{rem}Clearly every point in $\hat{X}$ has an open neighborhood isomorphic to an open subset of $X$. Hence every local property (e.g., reduced, normal, smooth) is inherited by it, and also by the given open subvariety.
\end{rem}

\begin{rem}\label{normal_qp}If $X$ is normal, then this $G$-invariant open subset $U$ can be chosen to be quasi-projective (and hence $Y$ can be chosen to be quasi-projective in Theorem \ref{main_reg}). Indeed then $\hat{X}$ is normal (since this is a local condition), and so is $U$. O.~Benoist \cite{B} proved that every normal variety has finitely many maximal open quasi-projective subvarieties (obviously, it has at least one). Applied inside $U$, we can then restrict to the intersection of those, which is a dense open $G$-invariant subset.

In general, when the variety $U$ is not normal, I do not know whether it always holds that $\Aut(U)$ preserves a dense open quasi-projective subset.
\end{rem}

\begin{rem}\label{ground_field}
That $\bfk$ is algebraically closed actually plays no role except for the exposition, and Theorem \ref{sep_core} works with an arbitrary field $\bfk$. Hence Theorem \ref{main_reg} holds in the category of $\bfk$-varieties (separated schemes of finite type over $\mathrm{Spec}(\bfk)$).
\end{rem}

\begin{rem}
For reduced varieties and assuming the field perfect, Theorem \ref{main_reg} has independently been obtained by Lonjou and Urech \cite{LU} using explicit constructions of CAT(0) cube complexes.
\end{rem}

\section{The second step: a general lemma}\label{gen_prel}

Recall that a topological space is {\bf noetherian} if every nonempty set of closed subsets has a minimal element for inclusion, or equivalently if there is no strictly decreasing sequence of closed subsets.

\begin{lem}\label{noeth_core}
Let $Y$ be a noetherian topological space. Let $G$ be a group acting continuously on $Y$. Let $X$ be a dense open subset of $Y$. Then there exists a dense $G$-invariant open subset $U$ of $Y$ such that for any $x,y\in U$ there exists $g\in G$ such that $(gx,gy)\in X^2$.
\end{lem}
\begin{proof}
For $x\in Y$, define
\[U_x=\big\{y\in Y:\exists g\in G:\{gx,gy\}\subset X\big\}.\]
Note that $y\in U_x$ if and only if $x\in U_y$. Also, $U_x$ is the union over all $g$ such that $gx\in X$, of $g^{-1}X$; hence $U_x$ is open. Moreover, $x\mapsto U_x$ is $G$-equivariant:
\[U_{hx}=\big\{y:\exists g:\{ghx,gy\}\subset X\big\}=\big\{y:\exists g:\{gx,gh^{-1}y\big\}\subset X\}=hU_x.\]
Let $F_x=Y\smallsetminus U_x$ be its complement (so $x\mapsto F_x$ is also $G$-equivariant), and let $F=Y\smallsetminus X$ be the complement of $X$.

For $x\in X$, clearly $X\subset U_x$, that is, $F_x\subset F$. Let $\mathcal{U}$ be the set of dense open subsets of $Y$; define
\[K=\bigcap_{V\in\mathcal{U}}\overline{\bigcup_{x\in V}F_x}.\]
As an intersection of closed subsets, $K$ is closed. Also $K\subset \overline{\bigcup_{x\in X}F_x}\subset F$. Moreover, $K$ is $G$-invariant, since $\mathcal{U}$ is $G$-invariant and using equivariance of $x\mapsto F_x$.

Since $Y$ is noetherian, there exists $V\in\mathcal{U}$ such that $\overline{\bigcup_{x\in V}F_x}=K$. Let $W$ be the open complement of $K$; it is $G$-invariant. Then for every $y\in W$ and every $x\in V$, we have $y\in U_x$. So, for every $y\in W$ and $x\in V$ we have $x\in U_y$. In other words, we have $V\subset \bigcap_{y\in W}U_y$. The latter intersection is $G$-invariant; let $U'$ be its interior, thus $V\subset U'$, so that $U'$ is dense and also $G$-invariant.

For all $x\in U'$ and all $y\in W$, we have $x\in U_y$. Hence this also holds for all $x,y\in U:=U'\cap W$. 
\end{proof}

\section{Neumann's lemma and partial actions}\label{sec_parta}

The following result plays a key role in the proof.

\begin{thm}[B.H. Neumann's lemma \cite{Neu}]\label{t_neumann}
Let $G$ be a group and $E$ a $G$-set. Let $F$ be a finite subset of $E$ not meeting any finite $G$-orbit in $E$. Then there exists $g\in G$ such that $F\cap gF$ is empty.
\end{thm}

Given a $G$-set $E$, a subset $X\subset E$ is said to be $G$-{\bf commensurated} if $X\smallsetminus g^{-1}X$ is finite for every $X$ (or equivalently $X\triangle gX=(X\smallsetminus gX)\sqcup g(X\smallsetminus g^{-1}X)$ is finite for all $g$). It is said to be $G$-{\bf transfixed} if there exists a $G$-invariant subset $Y$ such that $Y\triangle X$ is finite. We say that $X$ is $G$-{\bf transfixed above} if moreover $Y$ can be chosen to contain $X$. We have {\it (transfixed above)} $\Rightarrow$ {\it (transfixed)} $\Rightarrow$ {\it (commensurated)} and these are strict implications in general (for instance for the left-action of $\Z$ on itself, $\N$ is commensurated but not transfixed, and a singleton is not transfixed above but transfixed). We say that $X$ is {\bf finely $G$-transfixed above} if it is $G$-transfixed above, and moreover every finite $G$-orbit meeting $X$ is contained in $X$. Thus if $X$ is $G$-transfixed above, then there exists a finite subset $F$ such that $X\smallsetminus F$ is finely $G$-transfixed above.

 Let $G$ partially act on a set $X$ (so the universal globalization $\hat{X}$ will play the role of $E$). We say that $X$ is $G$-{\bf cofinite} if $X\smallsetminus D_g$ is finite for every $g\in G$. We say that $X$ is $G$-{\bf transfixed above} if $\hat{X}\smallsetminus X$ is finite. We say that $X$ is $G$-{\bf transfixed} if there exists a finite subset $F$ of $X$ such that $X\smallsetminus F$, with the restricted partial action, is $G$-transfixed above. We say that $X$ is {\bf finely $G$-transfixed above} if $\hat{X}\smallsetminus X$ is finite and meets no finite orbit.

It is straightforward to check that $X$ is $G$-cofinite if and only if $X$ is commensurated in $E$ (because the complement $X\smallsetminus D_g$ of the domain of $g\in G$ on $X$ is precisely $X\smallsetminus g^{-1}X$), and also the transfixing notions match. Therefore, Neumann's lemma (Theorem \ref{t_neumann}) can be translated as follows:

\begin{cor}
Let $G$ be a group and $X$ a partial $G$-set, that is finely transfixed above. Let $\hat{X}=\Compl(G,X)$ be the universal globalization. Then there exists $g\in G$ such that $\hat{X}=X\cup gX$.
\end{cor}
\begin{proof} Write $F=\hat{X}\smallsetminus X$. By assumption, $F$ is finite and intersects no finite orbit. By Theorem \ref{t_neumann} there exists $g\in G$ such that $F\cap gF=\emptyset$, which is exactly the required conclusion.
\end{proof}

\begin{defn}
A group $G$ has {\bf Property FW} if for every $G$-set $E$ and commensurated subset $X$, the subset $X$ is $G$-transfixed.
\end{defn}

Equivalently, $G$ has Property FW if for every partial $G$-set $X$, if $X$ is $G$-cofinite then $X$ is $G$-transfixed.

See \cite{CFW} for a detailed discussion of Property FW; the acronym ``FW" stands for ``Fixed point property on space with Walls". Examples of groups with Property FW are groups with Kazhdan's Property T, and others such as the group $\mathrm{SL}_2(\Z[\sqrt{k}])$, for $k\ge 2$ a positive non-square integer.

\section{The proof}

\begin{lem}\label{sep_criterion}
Let $X$ be a prevariety in which any two points belong to a common separated open subset. Then $X$ is separated.
\end{lem}
\begin{proof}
This is well-known. By assumption, the $U\times U$, for $U$ open separated subset, cover $X\times X$. Since being a closed immersion can be checked locally, we deduce that $X\to X\times X$ is a closed immersion.
\end{proof}

For a prevariety $X$, write $X=\bigsqcup_{i\ge 0}X_i$, where $X_i$ is the set of points that locally correspond to $i$-dimensional irreducible subvarieties. Thus $X_0$ is the union of all closed singletons. Write $X_{\ge i}=\bigcup_{j\ge i}X_j$.

\begin{proof}[Proof of Theorem \ref{sep_core}]

By reverse induction, for each $i\le d=\dim(X)$, we prove that there exists a finite subset $J$ of $G$ (of cardinal $\le 2^{d-i}$), and a dense open subset $Z$ of $X$ such that $\Big(\bigcup_{g\in J}gZ\Big)_{\ge i}$ is a $G$-invariant subset of the universal globalization $\hat{X}=\Compl(G,X)$.

This is clear for $i=d$, as $X_{\ge i}=X_i$ is an invariant singleton. Now, for $0\le i< d$ suppose that the claim is proved for $i+1$: there exists $J\subset G$ and a dense open subset $Z\subset X$, such that, denoting $Y=\bigcup_{g\in J}gZ$: the subset $Y_{\ge i+1}$ is $G$-invariant, with $|J|\le 2^{d-i-1}$.

Let $K$ be the complement of $Y$ in $\hat{X}$.
We claim that $hK_i\cap X_i$ is finite for every $h\in G$.
Indeed, $(hK\cap X)_{\ge i+1}$ being empty, $hK\cap X$ has dimension $\le i$ and is a closed subset of $X$, and hence $(hK\cap X)_i$ is indeed finite.

We next claim that $Y_i(=\bigcup_{g\in J}gZ_i)$ is $G$-commensurated. Indeed, for $g\in J$, we have $hK_i\cap gZ_i\subset hK_i\cap gX_i=g(g^{-1}hK_i\cap X)$ finite by the previous claim, and since it holds for each $g\in J$, we deduce that $hK_i\cap Y_i$ is finite, so that $Y_i$ is $G$-commensurated. 
 
Since $G$ has Property FW, $Y_i$ is $G$-transfixed. Hence there exists a finite subset $L$ of $Y_i$ such that $Y_i\smallsetminus L$ is finely transfixed above (as defined in \S\ref{sec_parta}). Thus, no finite orbit in $\hat{X}_i$ meets both $L$ and $Y_i$. 
Let $\dot{L}$ be the closed subset of $\hat{X}$ corresponding to $L$. Define $Z'=Z\smallsetminus\bigcup_{g\in J}g^{-1}\dot{L}$, and $Y'=\bigcup_{g\in J}gZ'\subset Y$. Then
\[Y\smallsetminus Y'=\bigcup_{h\in J}hZ\smallsetminus\left(\bigcup_{h\in J}hZ\smallsetminus\Big(\bigcup_{g\in J}hg^{-1}\dot{L}\Big)\right)\subset\bigcup_{h,g\in J}hg^{-1}\dot{L},\]
so $Y_i\smallsetminus Y'_i\subset GL$. The inclusions $Y'_i\subset Y_i\supset Y_i\smallsetminus L\subset \bigcup_{g\in G}g(Y_i\smallsetminus L)$ are all cofinite, and hence the inclusion $Y'_i\subset \bigcup_{g\in G}g(Y_i\smallsetminus L)$ is cofinite too, so $Y'_i$ is transfixed above. It is actually finely transfixed above, indeed, otherwise we find a finite orbit in $\hat{X}_i$ meeting $Y_i\smallsetminus Y'_i$, and hence meeting both $L$ and $Y_i$, a contradiction.

 By the corollary of Neumann's lemma, there exists $h\in G$ such that $Y'_i\cup hY'_i=\hat{Y'_i}$. That is, $Y'_i\cup hY'_i$ is $G$-invariant. For $j>i$, $Y'_j\cup hY'_j=Y'_j$ is $G$-invariant. So $(Y'\cup hY')_{\ge i}$ is $G$-invariant. Since $Y'\cup hY'=\bigcup_{g\in J'}gZ'$ with $J'=J\cup hJ$, the induction step is proved (with $|J'|\le 2|J|\le 2^{d-i}$).

Therefore, by the $i=0$ case of the claim, there exists a dense open subset $Z\subset X$ and a finite subset $J\subset G$ such that $Y=\bigcup_{g\in J}gZ$ is $G$-invariant. So $Y$ is noetherian. Hence we can apply Lemma \ref{noeth_core}: $Y$ has dense open subsets $U\subset Y'$ with $Y'$
$G$-invariant, $U$ separated, and every pair in $Y'$ can be $G$-translated into $U$; by Lemma \ref{sep_criterion} it follows that $Y'$ is separated.
\end{proof}

{\bf Acknowledgements.} I thank Serge Cantat for encouragement, valuable discussions and remarks on a preliminary version, notably the reference to \cite{CZ}. I owe to Michel Brion the argument using equivariant resolution of singularities in Remark \ref{firstrem}.

\end{document}